\documentclass[12pt,a4paper]{article}
\usepackage[utf8]{inputenc}
\usepackage[sc]{mathpazo}
\usepackage[T1]{fontenc}

\usepackage{amsmath,amsthm} 
\usepackage{latexsym,amssymb}
\usepackage{graphicx, epsfig}
\usepackage{mathtools}
\usepackage{array,booktabs}

\usepackage{fullpage}

\theoremstyle{plain} 
\newtheorem{thm}{Theorem}[section]
\newtheorem{cor}[thm]{Corollary}
\newtheorem{lem}[thm]{Lemma}
\newtheorem{prop}[thm]{Proposition}

\theoremstyle{definition}
\newtheorem{defn}[thm]{Definition}

\theoremstyle{remark}
\newtheorem{rem}[thm]{Remark}

\DeclarePairedDelimiter{\card}{\lvert}{\rvert} 

\newcommand{\mcf}{\mathcal{F}}
\newcommand{\mcg}{\mathcal{G}}

\newcommand{\mci}{\mathcal{I}}

\newcommand{\mcm}{\mathcal{M}}
\newcommand{\mcn}{\mathcal{N}}

\newcommand{\mcx}{\mathcal{X}}

\usepackage{color}
\definecolor{darkblue}{rgb}{0.0,0.0,0.3}
\definecolor{darkmagneta}{rgb}{0.5,0.0,0.5}
\definecolor{darkred}{rgb}{0.5,0.0,0.0}

\usepackage[bookmarks]{hyperref}
\hypersetup{colorlinks,breaklinks,
            linkcolor=darkblue,urlcolor=darkblue,
            anchorcolor=darkblue,citecolor=darkblue}

\newcommand\doilink[1]{\href{http://dx.doi.org/#1}{#1}}
\newcommand\doi[1]{doi:\doilink{#1}}

\renewcommand*\url[1]{\href{#1}{\texttt{#1}}}

\newcommand{\isp}{\overline{P}} 
\newcommand{\esp}{\overline{Q}} 
\newcommand{\cpl}{\overline{\Omega}} 

\title{A construction of the abstract induced subgraph poset of a graph from its abstract edge subgraph poset}

\author{Deisiane Lopes Gonçalves \thanks{Departamento de Matemática,
    Universidade Federal de Minas Gerais (UFMG), Brasil. %
    \textit{E-mail:} \texttt{deisianel@ufmg.br}} %
  \and %
  Bhalchandra D. Thatte \thanks{Departamento de Matemática, Universidade Federal
    de Minas Gerais (UFMG), Brasil. %
    \textit{E-mail:} \texttt{thatte@ufmg.br}}%
}

\begin{document}
\maketitle

\begin{abstract}
  The abstract induced subgraph poset of a graph is the isomorphism class of the
  induced subgraph poset of the graph, suitably weighted by subgraph counting
  numbers. The abstract bond lattice and the abstract edge-subgraph poset are
  defined similarly by considering the lattice of subgraphs induced by connected
  partitions and the poset of edge-subgraphs, respectively. Continuing our
  development of graph reconstruction theory on these structures, we show that
  if a graph has no isolated vertices, then its abstract bond lattice and the
  abstract induced subgraph poset can be constructed from the abstract
  edge-subgraph poset except for the families of graphs that we
  characterise. The construction of the abstract induced subgraph poset from the
  abstract edge-subgraph poset generalises a well known result in reconstruction
  theory that states that the vertex deck of a graph with at least 4 edges and
  without isolated vertices can be constructed from its edge deck.
  
  Mathematics Subject Classification MSC2020: 05C60
\end{abstract}

\section{Introduction}
\label{sec-intro}

One of the beautiful conjectures in graph theory, which has been open for more than 70 years, is about vertex reconstruction of graphs \cite{bondy2014}. It was proposed by Ulam and Kelly \cite{kelly1942}, and was reformulated by Harary \cite{hara1964} in the more intuitive language of reconstruction. This conjecture asserts that every finite simple undirected graph on three or more vertices is determined, up to isomorphism, by its collection of vertex-deleted subgraphs (called the deck). Harary \cite{hara1964} also proposed an analogous conjecture, known as the edge reconstruction conjecture. It asserts that every finite simple undirected graph with four or more edges is determined, up to isomorphism, by its collection of edge-deleted subgraphs (called the edge-deck). 

In a series of papers \cite{thatte-isp,thatte-esp,thatte-lattice}, the second
author considered three objects related to the vertex and the edge decks of a
graph $G$, namely, the abstract induced subgraph poset $\isp(G)$, the abstract bond lattice $\cpl(G)$, and the abstract edge-subgraph poset
$\esp(G)$. The induced subgraph poset $P(G)$ is the poset of distinct (mutually
non-isomorphic) induced non-empty subgraphs of $G$, where each edge of the
poset, say from $G_i$ to $G_j$, is weighted by the number of induced subgraphs
of $G_j$ that are isomorphic to $G_i$. By {\em abstract} we mean the isomorphism
class of $P(G)$, and denote it by $\isp(G)$. Analogously, the abstract bond lattice $\cpl(G)$ is defined for the distinct spanning subgraphs
induced by connected partitions of $G$, and the abstract edge-subgraph poset
$\esp(G)$ is defined for the distinct edge-subgraphs of $G$.

Many old and new results on the reconstruction conjectures were proved for the
problem of reconstructing a graph from its abstract induced subgraph poset or
from its abstract edge-subgraph poset. For example, it was shown in
\cite{thatte-isp} that all graphs on 3 or more vertices are vertex
reconstructible if and only if all graphs except the empty graphs are
$P$-reconstructible (reconstructible from the abstract induced subgraph
poset). On the other hand, even for classes of graphs known to be vertex
reconstructible (e.g., disconnected graphs), the problem of $P$-reconstruction
may be prohibitively difficult. Moreover, it was shown in \cite{thatte-lattice}
that $\cpl(G)$ and $\isp(G)$ can be constructed from each other except in a few
cases that were characterised. It was also shown in \cite{thatte-esp} that the
edge reconstruction conjecture is true if and only if all graphs (except graphs
from a particular family) are $Q$-reconstructible (reconstructible from the
abstract edge-subgraph poset). Here we continue to develop reconstruction theory
for the three posets.

In Section~\ref{sec-notation}, we define the notation and terminology for graphs
and posets, give formal definitions of the three posets $\isp(G), \esp(G)$ and
$\cpl(G)$, and state the corresponding reconstruction problems.

In Section~\ref{sec-main}, we prove the main theorem of the paper, namely,
Theorem~\ref{thmmain}, which states that $\isp(G)$ can be constructed from
$\esp(G)$, except when $G$ belongs to a certain family $\mcn$ of graphs, and
that $\cpl(G)$ can be constructed from $\esp(G)$, except when $G$ belongs to a
certain family $\mcm$ of graphs. The two families $\mcm$ and $\mcn$ differ
slightly, and we characterise them completely. The construction of $\isp(G)$
from $\esp(G)$ generalises a well known result in reconstruction theory that states that the vertex deck of a graph with at least 4 edges and without isolated vertices can be constructed from its edge deck \cite{hemminger.1969}.

\section{Notation and terminology}
\label{sec-notation}

In this section, we summarise the notational conventions and definitions from \cite{thatte-esp,thatte-lattice}, which we use with minor changes. 

\subsection{Graphs}

First we clarify our convention about unlabelled graphs. We identify each
isomorphism class of graphs (also called an {\em unlabelled graph}) with a
unique representative of the class. If $G$ is a graph, then $\overline{G}$ is
the unique representative of the isomorphism class of $G$. Let $\mcg$ be the set
consisting of one representative element from each isomorphism class of finite
simple graphs. We assume that all newly declared graphs are from $\mcg$. On the
other hand, if a graph $H$ is derived from a graph $G$ (e.g., $H = G-e$), then
$H$ may not be in $\mcg$. If $H\in \mcg$ is isomorphic to a subgraph of $G$, we say that $G$ contains $H$ as a subgraph or $G$ contains $H$ or $H$ is
contained in $G$.

An empty graph is a graph with empty edge set. A null graph $\Phi$ is a graph
with no vertices.

Let $G$ be a graph. We denote the number of vertices of $G$ by $v(G)$, the
number of edges of $G$ by $e(G)$, and the number of components of $G$ by $k(G)$,
the maximum degree in $G$ by $\Delta (G)$, and the minimum degree in $G$ by
$\delta (G)$.

For $X\subseteq V(G)$, we denote the subgraph induced by $X$ by $G[X]$, the subgraph induced by $V(G)\setminus X$ by $G-X$, or $G-u$ when
$X=\{u\}$. For $E\subseteq E(G)$, we denote the subgraph induced by $E$ by
$G[E]$, it is called an edge-subgraph. The spanning subgraph of $G$
with edge set $E(G)\setminus E$ by $G-E$, or $G-e$ if $E=\{e\}$.

We denote the number of subgraphs (induced subgraphs, edge-subgraphs) of $G$
that are isomorphic to $H$ by $s(H,G)$ ($p(H,G), q(H,G)$, respectively).

A path on $n$ vertices is denoted by $P_n$ and a cycle on $n$ vertices is
denoted by $C_n$. We denote a complete graph on $n$ vertices by $K_n$, a
complete bipartite graph with $n$ and $m$ vertices in two partitions by
$K_{n,m}$, and the graph $K_4$ minus an edge by $K_4 \setminus e$.

We refer to the following graphs in some proofs.
\begin{figure}[h]
  \centering
  \includegraphics{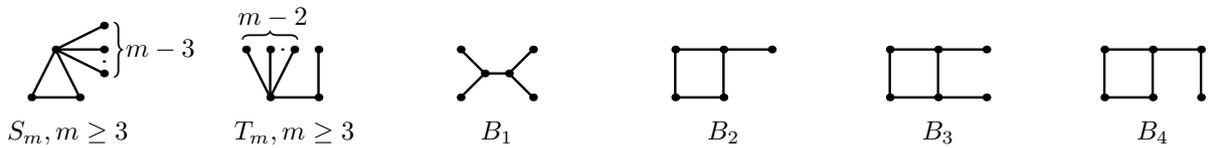}
  \caption{Some graphs referred to in proofs.}
  \label{fig:1}
\end{figure}

For terminology and notation about graphs not defined here, we refer to Bondy
and Murty \cite{bondy2008}.

\subsection{Partially ordered sets}

Let $(S,\leq)$ be a partially ordered set. If $x,y\in S$, $x\leq y$, $x\neq y$
and there is no $z\in S\setminus \{x,y\}$ such that $x\leq z\leq y$, then we say
$y$ covers $x$. We say that $\rho: S\to \mathbb{Z}$ is a rank function if for
all $x,y\in S$, $y$ covers $x$ implies $\rho(y)=\rho(x)+1$.

A weighted poset is a poset $(S, \leq)$ with a compatible weight function
$\omega: S\times S\to \mathbb{Z}$, where compatible means $\omega(x,y)=0$ unless
$x\leq y$. We say that weighted posets $(S,\leq,\omega)$ and
$(S',\leq',\omega')$ are isomorphic if there is a bijection $f: S\to S'$ such
that for all $x,y\in S$, we have $x\leq y$ if and only if $f(x)\leq'f(y)$ and
$\omega(x,y)=\omega'(f(x),f(y))$.

For terminology and notation about partially ordered sets, we refer to Stanley
\cite{stanley2011}.

\subsection{Graph posets}

\begin{defn}[Edge-subgraph poset] Let $Q \coloneqq (\mcg, \leq_e, q)$ be a
  weighted poset, where for all $H,G\in \mcg$, we have $H \leq_e G$ if $H$ is
  isomorphic to an edge subgraph of $G$, and $q:\mcg\times \mcg\to \mathbb{Z}$
  is a weight function such that for all $H,G\in \mcg$, $q(H,G)$ is the number
  of edge-subgraphs of $G$ that are isomorphic to $H$. For $G \in \mcg$, let
  $Q(G) \coloneqq \{H \in \mcg \mid H \leq_e G \text{ and } e(H) > 0\}$.  The
  \textbf{concrete edge-subgraph poset} of $G$ is the restriction of $Q$ to
  $Q(G)$; it is denote by just $Q(G)$. The \textbf{abstract edge-subgraph poset}
  of $G$ is the isomorphism class of $Q(G)$, which we identify with a weighted
  poset $\overline{Q}(G)\coloneqq (\{g_1,\ldots,g_m\},\leq_e,q)$, which is a
  representative concrete edge-subgraph poset in the isomorphism class of
  $Q(G)$. We assume that $g_1$ and $g_m$ are the minimal and the maximal
  elements in $\overline{Q}(G)$, respectively. We define a rank function
  $e\colon \overline{Q}(G) \to \mathbb{Z}$ such that $e(g_1)=1$ (so that
  $e(g_i)$ is the number of edges of $g_i$).
\end{defn}

\begin{defn}
  A graph $G$ is \textbf{$Q$-reconstructible} if it is determined up to
  isomorphism by its abstract edge-subgraph poset. An invariant of $G$ is said
  to be $Q$-reconstructible if it is determined by $\overline{Q}(G)$. A class of
  graphs is $Q$-reconstructible if each element of the class is
  $Q$-reconstructible.
\end{defn}

\begin{rem}
  For all graph $G$, we have $\overline{Q}(G)=\overline{Q}(G+K_1)$. Therefore,
  we understand $Q$-reconstructibility to mean $Q$-reconstructibility modulo
  isolated vertices.
\end{rem}

\begin{defn}[Induced subgraph poset] Let $P\coloneqq (\mcg, \leq_v, p)$ be a
  weighted poset, where for all $H,G\in \mcg$, we have $H \leq_v G$ if $H$ is
  isomorphic to an induced subgraph of $G$, and
  $p:\mcg\times \mcg\to \mathbb{Z}$ is a weight function such that for all
  $H,G\in \mcg$, $p(H,G)$ is the number of induced subgraphs of $G$ that are
  isomorphic to $H$.  For $G \in \mcg$, let
  $P(G) \coloneqq \{H \in \mcg \mid H \leq_v G \text{ and } e(H) > 0\} \cup
  \{K_1\}$. The \textbf{concrete induced subgraph poset} of $G$ is the
  restriction of $P$ to $P(G)$; it is denote by just $P(G)$. The
  \textbf{abstract induced subgraph poset} of $G$ is the isomorphism class of
  $P(G)$, which we identify with a weighted poset
  $\overline{P}(G)\coloneqq (\{g_1,\ldots,g_m\},\leq_v,p)$, which is a
  representative concrete induced subgraph poset in the isomorphism class of
  $P(G)$.
\end{defn}

\begin{defn}
  A graph $G$ is \textbf{$P$-reconstructible} if it is determined up to
  isomorphism by its abstract induced subgraph poset. An invariant of $G$ is
  said to be $P$-reconstructible if it is determined by $\overline{P}(G)$. A
  class of graphs is $P$-reconstructible if each element of the class is
  $P$-reconstructible.
\end{defn}

\begin{defn}
  Let $G$ be a graph. We say that a partition $\pi\coloneqq \{X_1,\ldots,X_n\}$
  of $V(G)$ is a \textbf{connected partition} of $V(G)$, if the induced
  subgraphs $G[X_1],\ldots, G[X_n]$ are connected; we write $\pi\vdash_c V(G)$.
\end{defn}

\begin{defn} \label{def-cpl} For $H_i, H_j\in\mcg$, if there exists
  $U\subseteq V(H_j)$ and $\pi\vdash_c U$ such that $H_j[\pi]\cong H_i$, we
  denote $H_i\leq_{\pi} H_j$.  Let $\Omega\coloneqq (\mcg, \leq_{\pi}, \omega)$
  be a weighted poset, where the weight function
  $\omega:\mcg\times \mcg\to \mathbb{Z}$ is such that for all $H,G \in \mcg$, we
  have
  $\omega(H,G)\coloneqq \card{\{(\pi, U)\mid U\subseteq V(G),\pi\vdash_c U\text{
      and } G[\pi]\cong H\}}$. For $G \in \mcg$, let
  $\Omega(G)\coloneqq \{H \in \mcg \mid H \leq_{\pi} G, v(H) = v(G)\}$. The
  \textbf{concrete bond lattice} of $G$ is the restriction of $\Omega$ to
  $\Omega(G)$; it is denote by just $\Omega(G)$. The \textbf{abstract bond
    lattice} of $G$ is the isomorphism class of $\Omega(G)$, which we identify
  with a weighted poset
  $\overline{\Omega}(G)\coloneqq (\{h_1,\ldots,h_m\},\leq_{\pi},\omega)$, which
  is a representative concrete bond lattice in the isomorphism class of
  $\Omega(G)$. We assume that $h_1$ and $h_m$ are the minimal and the maximal
  elements in $\overline{\Omega}(G)$, respectively. We define a rank function
  $\gamma$ on $\overline{\Omega}(G)$ such that $\gamma(h_1)=0$; hence
  $k(h_i)=v(G)-\gamma(h_i)$.
\end{defn}

Let $G$ be a graph. We consider elements of $\esp(G)$ as unknown graphs in
$\mcg$. A weight preserving isomorphism $\ell \colon \esp(G) \to Q(H)$ is called
a legitimate labelling of $\esp(G)$, and $H$ is called a $Q$-reconstruction of
$G$. (Note that $Q(H)$ is a concrete poset.) A graph invariant $f$ of $G$ is
$Q$-reconstructible if $f(H) = f(G)$ for all $Q$-reconstructions $H$ of $G$.
Let $x\in \esp(G)$. We say that $x$ is reconstructible or uniquely labelled (or
$x$ satisfies property $P$, or a graph invariant $f$ of $x$ is reconstructible)
if there is a graph $F$ such that for all legitimate labelling maps $\ell$ from
$\esp(G)$, we have $\ell(x) = F$ (or $\ell(x)$ satisfies property $P$, or
$f(\ell(x))$ is identical, respectively). Similarly, a graph $F$ is a
distinguished subgraph of $G$ (or simply $F$ is distinguished) if there is an
element $x$ such that for all legitimate labelling maps $\ell$ from $\esp(G)$,
we have $\ell(x) = F$. Analogous terminology may be used for $\isp(G)$ and
$\cpl(G)$.

\section{Constructing $\isp(G)$ and $\cpl(G)$ from $\esp(G)$}
\label{sec-main}
Let
\begin{align*}
  \mcf_0 & \coloneqq \{3K_2, K_3, K_{1,3}\}, \\
  \mcf_1& \coloneqq \{P_4, K_{1,2}+K_2, P_4+K_2, T_4 \}, \\
  \mcf_2& \coloneqq \{C_4, 2K_{1,2}, C_4+K_2, B_1, P_6, B_2, B_3,B_4\}, \\
  \mcf_3 &\coloneqq \{K_{1,m}\mid m>1 \text{ and } m\neq3\}\bigcup\{mK_2\mid m>1 \text{ and } m\neq3\}, \\
  \mcf_4 &\coloneqq \{rK_3+sK_{1,3}+F\mid r\neq s \text{ and } F\in\mathbb{N}^{\mcx}\}\setminus \{K_3,K_{1,3}\}, \text{ where}\\
  \mcx &\coloneqq\{P_n\mid n\geq 2\}\bigcup \{C_n\mid n\geq 4\}\bigcup \{S_4, K_4\setminus e, K_4\},
\end{align*} and $\mathbb{N}^{\mcx}$ is the set of all unlabelled finite graphs
(including the null graph) with components from $\mcx$.

The following result is the main theorem of this paper.

\begin{thm}[Main Theorem]\label{thmmain} Let $G$ be a graph with no isolated
vertices. Then
  \begin{enumerate}
  \item $\cpl(G)$ can be constructed from $\esp(G)$ if and only if $G$ does not
belong to $\mcm \coloneqq \mcf_0 \cup \mcf_2 \cup \mcf_4$.
  \item $\overline{P}(G)$ can be constructed from $\esp(G)$ if and only if $G$
does not belong to $\displaystyle \mcn \coloneqq \bigcup_{i=0}^4 \mcf_i$.
  \end{enumerate}
\end{thm}

Throughout this section we assume that $G$ is a graph without isolated vertices,
and that we are given $\esp(G)$. The main idea in proving Theorem~\ref{thmmain}
is the following lemma. When it is applicable, it allows us to recognise which
elements of $\esp(G)$ must be in $\cpl(G)$ and to calculate the weight function
in the definition of $\cpl(G)$.

\begin{lem}\label{lem-omega} Let $g_i,g_k \in \overline{Q}(G)$. Then
$\omega(g_i,g_k)$ may be expressed as a polynomial in $q(g_r,g_s)$, where $g_i
\leq_e g_r <_e g_s \leq g_k $ and, for all $g_r$, we have $v(g_i) = v(g_r)$ and
$k(g_i) = k(g_r)$, and for all $g_s\neq g_k$, we have $v(g_i) = v(g_s)$ and
$k(g_i) = k(g_s)$.
\end{lem}

\begin{proof} If $(v(g_i),k(g_i)) = (v(g_k),k(g_k))$, then
  \[ \omega(g_i,g_k) =
    \begin{cases} 1 & \text{ if } g_i = g_k, \\
      0 & \text{ otherwise}.
    \end{cases}
  \] Also, if $g_i \not \leq_e g_k$, then $\omega(g_i,g_k) = 0$. Hence in the
following calculation we assume that $(v(g_i),k(g_i)) \neq (v(g_k),k(g_k))$ and
$g_i <_e g_k$. We have
\begin{equation}\label{eq-omega-1}
  q(g_i,g_k)=\sum_{\substack{g_j \mid v(g_j) = v(g_i), \\
        k(g_j) = k(g_i)}}q(g_i,g_j)\omega(g_j,g_k).
  \end{equation} We rewrite Equation~(\ref{eq-omega-1}) as
  \begin{equation}\label{eq-omega-2}
    \omega(g_i,g_k)= q(g_i,g_k) - \sum_{\substack{g_j\mid g_i <_e g_j \\
        v(g_j) = v(g_i), \\
        k(g_j) = k(g_i)}} q(g_i,g_j)\omega(g_j,g_k),
  \end{equation} and repeatedly expand the factors $\omega(g_j,g_k)$ in each
term on the right hand side, with the condition that $\omega(g_j,g_k) =
q(g_j,g_k)$ if there is no $g_r$ such that $g_j <_e g_r <_e g_k$ and $v(g_r) =
v(g_j)$ and $k(g_r) = k(g_j)$. Thus we obtain the required polynomial.
\end{proof}

Lemma~\ref{lem-omega} can be used to calculate $\omega(g_i,g_k)$ only if we know
the number of vertices and the number of components of all elements that appear
in the computation. Hence most of the following lemmas are meant to either
$Q$-reconstruct $G$ or to show that the number of vertices and the number of
components of all elements that appear in the computation of $\omega(g_i,g_k)$
can be reconstructed.

We use several results from \cite{thatte-esp,thatte-lattice} in the proof.

\begin{thm}[Theorem 2.1 \cite{thatte-esp}] \label{thm-q-erc} The graphs in
  $\mcn$ are not $Q$-reconstructible. The edge reconstruction conjecture is true
  if and only if all graphs, except the graphs in $\mcn$, are
  $Q$-reconstructible.
\end{thm}

\begin{lem}[Corollary 2.11 \cite{thatte-esp}] \label{lem1.0} Acyclic graphs
(i.e., trees and forests) that are not in $\mcn$ are $Q$‐reconstructible.
\end{lem}

\begin{lem}[Lemma 2.7 \cite{thatte-esp}] \label{lem-k1m} For all $m\geq 4$, if
$G$ belongs to the class $\{K_{1,m}^{+e}\}\setminus\{K_{1,m+1}\}$, then $G$ is
$Q$-reconstructible, and $K_{1,3}$ and $K_3$ (if $G$ contains $K_3$) are
distinguished.
\end{lem}

\begin{lem}[Lemma 2.8 \cite{thatte-esp}] \label{lem-mk2} For all $m\geq 4$, if
$G$ belongs to the class $\{(mK_2)^{+e}\}\setminus\{(m+1)K_2\}$, then $G$ is
$Q$-reconstructible, and $3K_2$ is distinguished.
\end{lem}

\begin{lem}[Proposition 2.9 \cite{thatte-esp}] \label{lem-7edges} Graphs with at
most 7 edges, except the ones in $\mcn$, are $Q$-reconstructible.
\end{lem}

\begin{thm}[Theorem 3.7 \cite{thatte-lattice}] \label{thm-cpl-p} If $G$ does not
belong to $\{K_{1,n}(n>1),nK_2(n>1),P_4,K_{1,2}+K_2,T_4,P_4+K_2\}$, then
$\overline{P}(G)$ can be constructed from $\cpl(G)$.
\end{thm}

Next we prove several propositions and lemmas that eventually imply the main
theorem. We will first prove part (1) of the theorem, and then use
Theorem~\ref{thm-cpl-p} to prove the second part.

\begin{prop}[The ``only if'' parts of the main theorem] \label{prop-onlyif} \text{}
  \begin{enumerate}
  \item If $G$ belongs to $\mcm$, then $\cpl(G)$ cannot be constructed from
$\esp(G)$.
  \item If $G$ belongs to $\mcn$, then $\isp(G)$ cannot be constructed from
$\esp(G)$.
  \end{enumerate}

\end{prop}

\begin{proof}
  \begin{enumerate}
  \item The graphs $3K_2, K_{1,3}, K_3$ have the same $\overline{Q}(.)$, and the
    graphs $3K_2, K_{1,3}$ have the same $\overline{\Omega}(.)$, but
    $\overline{\Omega}(K_3)$ is different.

    If $G\in \mcf_4$, then suppose that $G = rK_3+sK_{1,3}+F$ and
    $H=sK_3+rK_{1,3}+F$, where $r\neq s$. Now $\cpl(G)\neq \cpl(H)$: this
    follows from the fact that the rank of the maximal element in $\cpl(.)$ is
    $v(.) - k(.)$, hence the two posets have different heights.
    
    Graphs in $\mcf_2$ form pairs so that graphs in each pair have the same
    abstract edge subgraph poset but different abstract bond lattice.

  \item If $G,H \in \mcn$ such that $\esp(G)\cong \esp(H)$ and $G \not\cong H$,
    then there are only two cases in which $v(G) = v(H)$, which are
    $\{G,H\} = \{C_4+K_2, B_1\}$ and $\{G,H\} = \{B_2,B_3\}$. In these cases, we
    verify that $\isp(G) \not\cong \isp(H)$. In all other cases,
    $v(G)\neq v(H)$, hence $G$ and $H$ cannot have the same abstract induced
    subgraph poset. Hence if $G\in \mcn$, then $\isp(G)$ cannot be constructed
    from $\esp(G)$.\qedhere
  \end{enumerate}
\end{proof}

\begin{prop}\label{p-f1f3} If $G\in \mcf_1 \cup \mcf_3$, then $\cpl(G)$ can be
  constructed from $\esp(G)$.
\end{prop}

\begin{proof} Graphs in each pair $\{K_{1,m}, mK_2\}, m \neq 3$ have the same
  $\esp(.)$ and the same $\cpl(.)$. We prove that no other graph has the same
  $\esp(.)$ as one of these graphs by induction on the number of edges of
  $G\in \{K_{1,m}, mK_2\}$. If $e(G)=2$ or $e(G)=4$ then the result is
  true. Suppose that the result is true for all $G\in \{K_{1,m}, mK_2\}$ such
  that $m\geq 4$. If $e(G)=m+1$, then by induction hypothesis, there is an
  element $x\in\esp(G)$ such that $x$ represents $K_{1,m}$ or $mK_2$. By Lemmas
  \ref{lem-k1m} and \ref{lem-mk2}, we have $G\in \{K_{1,m+1}, (m+1)K_2\}$.

  Graphs in each pair $\{P_4, K_{1,2}+K_2\}$, $\{P_4+K_2, T_4\}$ have the same
  $\esp(.)$ and the same $\cpl(.)$. By Lemma \ref{lem-7edges}, no other graph
  has the same $\esp(.)$ as one of these graphs.
\end{proof}

\begin{lem} \label{lem-vk} If $G \not \in \mcn$, then $v(G)$ and $k(G)$ are
  $Q$-reconstructible.
\end{lem}

\begin{proof} We prove the result by induction on the number of edges of
  $G$. All graphs on at most 7 edges that are not in $\mcn$ are
  $Q$-reconstructible (Lemma~\ref{lem-7edges}). Hence we take $e(G) = 7$ as the
  base case, for which the result is true. Suppose now that the result is true
  for all $G\not\in \mcn$ such that $7 \leq e(G) \leq m$. Let $G \not \in \mcn$
  be a graph on $m+1$ edges.
  
  If $G$ is an acyclic graph, then $G$ is $Q$-reconstructible (Lemma
  \ref{lem1.0}), thus $k(G),v(G)$ are known. So assume that $G$ contains a
  cycle. Hence each edge $e$ in $G$ that is on a cycle is such that
  $k(G-e) = k(G)$ and $v(G-e) = v(G)$ and $G-e$ has no isolated vertices (since
  we have assumed that $G$ has no isolated vertices).

  Now consider an arbitrary element $x$ of rank $m$ in $\esp(G)$. Suppose
    that $x \coloneqq \overline{G-e}$ (minus the resulting isolated
    vertices). We claim that $x$ cannot be in $\mcf_0 \cup \mcf_1\cup \mcf_2$ since
    $m \geq 7$ and all graphs in $\mcf_0\cup \mcf_1\cup \mcf_2$ have at most 6 edges. We
    also assume that $x$ cannot be in $\mcf_3$ since in that case $G$ would be
    $Q$-reconstructible (by Lemma~\ref{lem-k1m} and \ref{lem-mk2}). If
    $x \not\in \mcn$, then we know $v(x)$ and $k(x)$ by induction hypothesis;
    and we annotate $x$ with this information.  If it is known from $\esp(G)$
    that $G$ has an edge $e$ such that $e$ is on a cycle and $\overline{G-e}$ is
    not in $\mcf_4$, then $v(G)$ is the maximum $v(H)$ among all graphs of rank
    $m$ in $\esp(G)$ that are not in $\mcn$. Next we show that such an edge must
    exist.  

   Consider a graph $H \in \mcf_4$. Suppose that $G = \overline{H+uv}$, where
    $u$ and $v$ are non-adjacent vertices in the same component of $H$ (so that
    $uv$ is on a cycle in $H+uv$). Each component of $H$ is in
    $\mcx \cup \{K_{1,3}, K_3\}$. We verify that for all possible ways of adding
    an edge $uv$ between two vertices of the same component of $H$, we either
    obtain a graph in $\mcn$ or obtain a graph $G$ which contains an edge $e$ on
    a cycle such that $\overline{G-e}$ is not in $\mcn$. Since $G\not \in \mcn$,
    it is the latter case. Thus $v(G)$ is reconstructible.

  Now, $k(G)$ is equal to the minimum number of components among graphs not in
  $\mcn$ that correspond to elements of rank $m$ and that have the same number
  of vertices as $G$.
\end{proof}

\begin{lem}[Lemma 2.5 \cite{thatte-esp}] \label{lem-k12+2k2} The graph
  $K_{1,2}+2K_2$ is $Q$-reconstructible, and all elements of
  $\esp(K_{1,2}+2K_2)$ are uniquely labelled.
\end{lem}

\begin{lem} \label{lem-t4-k12+2k2} If $G\not \in \mcn$ and contains both
  $K_{1,2}+2K_2$ and $T_4$ as subgraphs, then $v(x)$ and $k(x)$ are
  reconstructible for all $x$ in $\esp(G)$.
\end{lem}

\begin{proof} Let $x$ be an element of $\esp(G)$. If $x\not\in \mcn$, then
  $v(x)$ and $k(x)$ are known by Lemma~\ref{lem-vk}. Hence assume that
  $x\in \mcn$.
  
  We claim that all elements of $\esp(G)$ of rank $3$ are uniquely
  labelled. Indeed, all elements of $\esp(K_{1,2}+2K_2)$ are uniquely labelled
  (by Lemma \ref{lem-k12+2k2}), $\esp(T_4)\cong \esp(P_4+K_2)$, and there is no
  other graph $H$ such that $Q(H)\cong Q(T_4)$, we distinguish $T_4$. Now
  $K_{1,3}$ is distinguished since $T_4$ contains $K_{1,3}$, but does not
  contain $K_3$ or $3K_2$. Furthermore, $3K_2$ is distinguished. Thus $K_3$ is
  distinguished also. Other subgraphs with 3 edges are distinguished since they
  are edge reconstructible, and graphs with 2 edges are distinguished. Thus, the
  graphs in $\mcf_0$, $P_4+K_2$ and $T_4$ are distinguished.
  
  Graphs $P_4$ and $K_{1,2}+K_2$ are distinguished since $K_{1,2}+2K_2$ contains
  $K_{1,2}+K_2$, but does not contain $P_4$.

  Pairs of graphs in $\mcf_2$ are distinguished since $P_4$ and $K_{1,3}$ are
  distinguished.

  Graphs $K_{1,m}$ and $mK_2$ are distinguished since $K_{1,m}$ contains
  $K_{1,3}$ as a subgraph, while $mK_2$ does not.

  Graphs $rK_3+sK_{1,3} + F$ and $sK_3+rK_{1,3} + F$, where $r \neq s$, are distinguished since
  $K_{1,3}$ and $K_3$ are distinguished.
\end{proof}

\begin{lem} \label{lem-atleast4} If $G \not \in \mcn$, $e(G) \geq 5$, and
  $\Delta(G)\geq 4$, then $v(x)$ and $k(x)$ are reconstructible for all $x$ in
  $\esp(G)$.
\end{lem}

\begin{proof} Let $x \in \esp(G)$. If $x\not\in \mcn$, then $v(x)$ and $k(x)$
  are known by Lemma~\ref{lem-vk}. Hence assume that $x\in \mcn$. By
  Lemma~\ref{lem-k1m}, $K_{1,3}$ is distinguished. This allows distinguishing
  subgraphs isomorphic to $P_4+K_2, T_4, B_1, C_4+K_2, B_2, P_6, B_3, B_4$ and
  all subgraphs in $\mcf_3\cup\mcf_4$. Hence for such elements, $v(x)$ and
  $k(x)$ are known. The only elements in $\mcn$ that are not yet distinguished
  are $P_4, K_{1,2}+K_2, C_4, 2K_{1,2}$.

  We have $\Delta(G)\geq 4$ and $G$ itself is not $K_{1,\Delta(G)}$, hence $G$
  contains one of the graphs in $\{K_{1,4}^{+e}\}\setminus \{K_{1,5}\}$ (i.e.,
  one of $K_{1,4}+K_2, S_5, T_5$) as a subgraph. We use such a subgraph to
  distinguish $P_4$ and $K_{1,2}+K_2$.

  If $G$ contains $K_{1,4}+K_2$, then $K_{1,4}+K_2$ is distinguished (by
  Lemma~\ref{lem-k1m}). Now $K_{1,2}+K_2$ and $P_4$ are distinguished since
  $K_{1,4}+K_2$ does not contain $P_4$ but contains $K_{1,2}+K_2$.
  
  If $G$ contains $S_5$, then it also contains $S_4$, and $S_4$ is
  $Q$-reconstructible. Hence $S_4$ is distinguished. Now $K_{1,2}+K_2$ and $P_4$
  are distinguished since $S_4$ does not contain $K_{1,2}+K_2$ but contains 2
  subgraphs isomorphic to $P_4$.

  If $G$ does not contain $S_4$ but contains $T_5$, then $G$ contains $T_4$,
  which is distinguished (as noted above). Now $P_4$ and $K_{1,2}+K_2$ are
  distinguished since $T_4$ contains one subgraph isomorphic to $K_{1,2}+K_2$
  and two subgraphs isomorphic to $P_4$.

  Once $P_4$ is distinguished, we also distinguish $C_4$ and $2K_{1,2}$ since
  only $C_4$ contains $P_4$ as a subgraph.

  Now $v(x)$ and $k(x)$ are known for all $x$ in $\esp(G)$.
\end{proof}

\begin{lem} \label{lem-K_1,2+2K_2-notT_4} If $G\notin \mcn$, $\Delta(G) \leq 3$,
  and $G$ contains $K_{1,2}+2K_2$ but does not contain $T_4$, then $G$ is
  $Q$-reconstructible.
\end{lem}

\begin{proof} Let $\mci$ be the class of graphs satisfying the conditions in the
  statement of the lemma.

  First we show that we can recognise if $G$ is in $\mci$. Graphs not in $\mcn$
  with at most 7 edges are $Q$-reconstructible (Lemma~\ref{lem-7edges}), so we
  assume that $e(G) > 7$.  Since $G$ is not $K_{1,m}$ for any $m$, we can assume
  that $\Delta(G) < e(G)$. By Lemma~\ref{lem-k1m}, we can recognise if
  $\Delta(G) \geq 4$.  Since $K_{1,2}+2K_2$ is $Q$-reconstructible, we can
  recognise if $G$ contains $K_{1,2}+2K_2$ as a subgraph.  We can recognise if
  $G$ contains $T_4$ as a subgraph as follows: if $G$ contains $T_4$, then $G$
  has a 7-edge subgraph $F$ that contains $T_4$, and thus cannot be in
  $\mcn$. Hence $F$ is $Q$-reconstructible (by Lemma~\ref{lem-7edges}). Thus we
  know that $G$ contains $T_4$ as a subgraph. Hence we now assume that $G$ and
  all its $Q$-reconstructions are in $\mci$.

  Next we make some observations about the structure of any graph $G$ in class
  $\mci$. Graphs $S_4, K_4\setminus e, K_4$ are in $\mci$, and since $G$ does
  not contain $T_4$, these graphs can only occur as subgraphs of 4-vertex
  components. Also, we can verify that any subgraph isomorphic to $K_3$ or
  $K_{1,3}$ is either a component of $G$ or a subgraph of a component isomorphic
  to $S_4$, $K_4\setminus e$ or $K_4$. All other components of $G$ are paths or
  cycles. Hence $G$ is of the form $rK_3 + sK_{1,3} + F$, where components of
  $F$ are in $\mcx$. But $G$ is not in $\mcn$, hence $r = s$. Since
  $q(K_{1,3}, H) = q(K_3,H)$ for all $H\in \{S_4, K_4\setminus e, K_4\}$, we
  have $k(K_3,G) = k(K_{1,3},G)$ and $q(K_3,G) = q(K_{1,3},G)$ for all graphs in
  $\mci$.

  Graphs $S_4, K_4\setminus e, K_4$ are all $Q$-reconstructible, hence
  $q(S_4,G)$, $q(K_4\setminus e, G)$ and $q(K_4,G)$ are known. Therefore,
  $k(S_4,G)$, $k(K_4\setminus e, G)$, $k(K_4,G)$, $k(K_3,G)$ and $k(K_{1,3},G)$
  are all reconstructible by the following sequence of calculations:
  \begin{align*}
    k(K_4,G) &= q(K_4,G)\\
    k(K_4\setminus e, G) &= q(K_4\setminus e, G) - q(K_4\setminus e, K_4)k(K_4,G)\\
    k(S_4,G) &= q(S_4,G) - q(S_4, K_4\setminus e)k(K_4\setminus e, G) - q(S_4, K_4)k(K_4, G),\\
    k(K_3,G) = k(K_{1,3},G) &= q(K_3,G) - k(S_4,G) - 2k(K_4\setminus e,G) - 4k(K_4,G).
  \end{align*}
  
  The graph $K_{1,2}+2K_2$ is contained in $G$, and by Lemma \ref{lem-k12+2k2},
  the graphs $K_2$, $P_3$, $2K_2$, $K_{1,2}+K_2$, $3K_2$, and $K_{1,2}+2K_2$ are
  distinguished. Now $P_4$ is distinguished because it is edge reconstructible
  and all edge-deleted subgraphs of $P_4$ are distinguished. The argument
  extends to $C_4$, $P_5$, $C_5$, $P_6$ and $C_6$ - they are all
  distinguished. (Of these graphs, $C_5$, $C_6$ and $P_5$ are
  $Q$-reconstructible.) Thus we know $q(P_i,G), 2 \leq i \leq 6 $ and
  $q(C_i,G), 4 \leq i \leq 6$.

  Graphs with at most 7 edges that are not in $\mcn$ are $Q$-reconstructible,
  hence we prove the lemma by induction on $e(G)$, with $e(g) = 7$ as the base
  case. Suppose that the result is true for all $G\in\mci$ such that
  $7 \leq e(G) \leq m$. Let $G$ be a graph with $e(G) = m+1$, and suppose that
  we are given $\esp(G)$.

  Paths and cycles with $i$ edges, for $7 \leq i \leq m$ are distinguished by
  induction hypothesis. Hence if $G$ itself is a cycle or path, then its
  edge-deck is known, and by edge reconstructibility of cycles and paths, $G$ is
  $Q$-reconstructible. Hence we assume that $G$ is not a path or cycle.

  Cycles of length 5 or more can only be components (since $G$ does not contain
  $T_4$). Now, we calculate $k(P_i,G), 2 \leq i \leq m$ and
  $k(C_i,G), 4 \leq i \leq m$ by solving the following equations (in that
  order):
  \begin{align*}
    k(C_4,G) &= q(C_4,G) - q(C_4,K_4\setminus e)k(K_4\setminus e,
               G) - q(C_4,K_4)k(K_4, G), \\
    k(C_i,G) &= q(C_i,G) \text{ for } 5 \leq i \leq m, \\
    k(P_i,G) &= q(P_i,G) - \sum_{H} q(P_i,H)k(H,G) \text{ for } i = m, m-1, \ldots, 2,
  \end{align*} where the summation in the last equation is over $H \in
  \{K_{1,3}, K_3, S_4, K_4\setminus e, K_4, P_r, r > i, C_s, s \geq 4 \}$.

  Now all components of $G$ are known along with their multiplicities,
  completing the induction step and the proof.
\end{proof}

\begin{lem} \label{lem-T_4-not-K_1,2+2K_2} If $G\notin \mcn$, $\Delta(G)\leq 3$,
  and $G$ contains $T_4$ but does not contain $K_{1,2}+2K_2$, then $G$ is
  $Q$-reconstructible.
\end{lem}

\begin{proof} Graphs with at most 7 edges that are not in $\mcn$ are
  $Q$-reconstructible, hence we assume that $e(G)>7$.

  If a graph contains $T_4$, and has 2 or more non-trivial components, then it also contains
$K_{1,2}+2K_2$. Therefore we assume that $G$ is connected. The conditions on $G$ also imply that $v(G) \in \{5,6\}$. If $v(G) = 5, \Delta(G) \leq 3$, then $e(G) \leq 7$ and the result is true. Therefore, we can assume that $v(G) = 6$. Now $\Delta(G) \leq 3$ and $e(G) > 7$ imply that have $e(G) \in \{8,9\}$.
  
  For any graph $H$, if $4 \leq e(H) \leq 9$, then $v(H)\leq 10$ or $H$ is
  disconnected. In either case, $H$ is vertex reconstructible, and hence edge
  reconstructible (see \cite{hemminger.1969,kelly.1957,mckay.1977}). If
  $e(G)=8$, the degree sequence of $G$ is $3,3,3,3,2,2$ or $3,3,3,3,3,1$. Now a
  $7$-edge subgraph of $G$ cannot have a component isomorphic to $K_3$ or
  $K_{1,3}$ (since the degree sequence of $K_3$ and $K_{1,3}$ is $2,2,2$ and
  $3,1,1,1$, respectively), hence $G$ does not have a 7-edge subgraph in
  $\mcn$. Hence all edge-deleted subgraphs of $G$ are $Q$-reconstructible, and
  since $G$ is edge reconstructible, it is also $Q$-reconstructible.

  If $e(G)=9$, then the degree sequence of $G$ is $3,3,3,3,3,3$, by a similar
  argument as in the case $e(G) = 8$, we note that no edge deleted subgraph of
  $G$ is in $\mcn$, hence using the case $e(G) = 8$, we can construct the edge
  deck of $G$, and then reconstruct $G$ since $G$ is edge reconstructible.
\end{proof}

\begin{proof} \textit{of Theorem \ref{thmmain} (the 'if' part)} Let $G$ be a
  graph such that $G\notin \mcm$.

  If $G\notin \mcn$ and $e(G)\leq 7$, then $G$ is $Q$-reconstructible by Lemma
  \ref{lem-7edges}. Hence $\cpl(G)$ can be constructed from $\esp(G)$. If
  $G\notin \mcn$, and contains $K_{1,2}+2K_2$ or $T_4$, and $\Delta(G)\leq 3$,
  then $G$ is $Q$-reconstructible by Lemmas \ref{lem-K_1,2+2K_2-notT_4} and
  \ref{lem-T_4-not-K_1,2+2K_2}. Hence $\cpl(G)$ can be constructed from
  $\esp(G)$.  If $G$ contains neither $T_4$ nor $K_{1,2}+2K_2$ as a subgraph,
  then $e(G)\leq 7$.  If $G\in \mcf_3\cup \{P_4, K_{1,2}+K_2, P_4+K_2, T_4\}$,
  then $\cpl(G)$ can be constructed from $\esp(G)$ by Propositions \ref{p-f1f3}.
  In all other cases, by Lemmas~\ref{lem-t4-k12+2k2} and~\ref{lem-atleast4}, we
  can reconstruct $v(x)$ and $k(x)$ for all $x\in\esp(G)$.

  Now we apply Lemmas~\ref{lem-omega} to $g_i$ and $g_k$, with $g_k$ as the
  maximal element, to recognise if $g_i$ must be $\cpl(G)$. Then we again apply
  Lemma~\ref{lem-omega} for all $g_i,g_k$ that are marked as elements of
  $\cpl(G)$.

  Let $G$ be a graph such that $G\notin \mcn$. Given $\esp(G)$, we construct
  $\cpl(G)$. Now, by Theorem \ref{thm-cpl-p}, we construct $\overline{P}(G)$
  from $\cpl(G)$.
\end{proof}

The following corollary generalises the result that edge reconstruction
conjecture is weaker than the vertex reconstruction conjecture.

\begin{cor} If a graph $G$ not in $\mcn$ is $P$-reconstructible, then it is
  $Q$-reconstructible.
\end{cor}

\section*{Acknowledgements} This study was financed in part by the Coordenação
de Aperfeiçoamento de Pessoal de Nível Superior – Brasil (CAPES) – Finance Code
001. The first author would like to thank CAPES for the doctoral scholarship at
UFMG.

\bibliographystyle{amsplain}

\begin{thebibliography}{10}

\bibitem{bondy2014}
J.~A. Bondy, \emph{Beautiful conjectures in graph theory}, European Journal of
  Combinatorics \textbf{37} (2014), 4–23.

\bibitem{bondy2008}
J.~A. Bondy and U.~S.~R. Murty, \emph{{Graph Theory}}, Graduate Texts in
  Mathematics, vol. 244, Springer, New York, 2008. \MR{2368647 (2009c:05001)}

\bibitem{hara1964}
F.~Harary, \emph{On the reconstruction of a graph from a collection of
  subgraphs}, Theory of Graphs and its applications (M. Fiedler,ed) (1964),
  47--52.

\bibitem{hemminger.1969}
Robert~L. Hemminger, \emph{On reconstructing a graph}, Proc. Amer. Math. Soc.
  \textbf{20} (1969), 185--187. \MR{MR0232696 (38 \#1019)}

\bibitem{kelly1942}
P.~J. Kelly, \emph{On isometric transformations}, Ph.D. thesis, University of
  Wisconsin. (1942).

\bibitem{kelly.1957}
Paul~J. Kelly, \emph{A congruence theorem for trees}, Pacific J. Math.
  \textbf{7} (1957), 961--968. \MR{0087949 (19,442c)}

\bibitem{mckay.1977}
B.~D. McKay, \emph{Computer reconstruction of small graphs}, J. Graph Theory
  \textbf{1} (1977), no.~3, 281--283. \MR{MR0463023 (57 \#2987)}

\bibitem{stanley2011}
Richard~P. Stanley, \emph{Enumerative combinatorics: Volume 1}, 2nd ed.,
  Cambridge University Press, USA, 2011.

\bibitem{thatte-isp}
Bhalchandra~D. Thatte, \emph{Kocay's lemma, {W}hitney's theorem, and some
  polynomial invariant reconstruction problems}, Electron. J. Combin.
  \textbf{12} (2005), no.~4, R63, 30 pp. (electronic). \MR{2180800
  (2006f:05121)}

\bibitem{thatte-esp}
\bysame, \emph{The edge-subgraph poset of a graph and the edge reconstruction
  conjecture}, J. Graph Theory \textbf{92} (2019), no.~3, 287--303.
  \MR{4009306}

\bibitem{thatte-lattice}
\bysame, \emph{The connected partition lattice of a graph and the
  reconstruction conjecture}, J. Graph Theory \textbf{93} (2020), no.~2,
  181--202. \MR{4043753}

\end{thebibliography}

\providecommand{\bysame}{\leavevmode\hbox to3em{\hrulefill}\thinspace}
\providecommand{\MR}{\relax\ifhmode\unskip\space\fi MR }
\providecommand{\MRhref}[2]{%
  \href{http://www.ams.org/mathscinet-getitem?mr=#1}{#2}
}
\providecommand{\href}[2]{#2}

\end{document}